\documentclass{amsart}
\usepackage{color,graphicx, verbatim, amsmath, amssymb, amsthm, amsfonts, 
mathtools}	
\usepackage[colorlinks=true,citecolor=black,linkcolor=black,urlcolor=blue]{hyperref}

\usepackage{ifthen}
\newtheorem{proposition}{Proposition}[section]
\newtheorem{theorem}[proposition]{Theorem}

\newtheorem{lemma}[proposition]{Lemma}
\newtheorem{prop}[proposition]{Proposition}
\newtheorem{cor}[proposition]{Corollary}

\theoremstyle{definition}
\newtheorem{example}[proposition]{Example}

\theoremstyle{remark}
\newtheorem{remark}[proposition]{Remark}

\numberwithin{equation}{section}

\newcommand{\arxiv}[1]{\href{http://arxiv.org/abs/#1}{\texttt{arXiv:#1}}}

\newcommand{\newword}[1]{\textbf{\emph{#1}}}

\newcommand{\from}{\leftarrow}

\newcommand{\Seed}{\operatorname{Seed}}
\newcommand{\Camb}{\operatorname{Camb}}

\newcommand{\sgn}{\operatorname{sgn}}
\renewcommand{\int}{\operatorname{int}}

\newcommand{\cov}{\mathrm{cov}}

\newcommand{\covered}{{\,\,<\!\!\!\!\cdot\,\,\,}}
\newcommand{\set}[1]{{\lbrace #1 \rbrace}}

\newcommand{\pidown}{\pi_\downarrow}

\newcommand{\br}[1]{{\langle #1 \rangle}}
\newcommand{\g}{{\mathbf g}}

\renewcommand{\c}{{\mathbf c}}

\newcommand{\A}{{\mathcal A}}

\newcommand{\C}{{\mathcal C}}

\newcommand{\tB}{{\widetilde{B}}}

\newcommand{\F}{{\mathcal F}}

\newcommand{\ck}{^{\vee\!}}

\newcommand{\ZZ}{\mathbb{Z}}

\newcommand{\RR}{\mathbb{R}}

\newcommand{\Tits}{\mathrm{Tits}}
\newcommand{\Cone}{\mathrm{Cone}}

\newcommand{\Ex}{\mathrm{Ex}}

\DeclareMathOperator{\inv}{inv}

\DeclareMathOperator{\Cart}{Cart}

\newcommand{\DF}{{\mathcal {DF}}}

\newcommand{\DCamb}{\operatorname{DCamb}}
\newcommand{\DC}{\mathrm{DC}}

\begin{document}

\title{A Cambrian framework for the oriented cycle}
\author{Nathan Reading and David E Speyer}

\thanks{Nathan Reading was partially supported by NSF grant DMS-1101568.   David E Speyer was supported in part by a Clay Research Fellowship}

\begin{abstract}
This paper completes the project of constructing combinatorial models (called frameworks) for the exchange graph and $\g$-vector fan associated to any exchange matrix $B$ whose Cartan companion is of finite or affine type, using the combinatorics and geometry of Coxeter-sortable elements and Cambrian lattices/fans. 
Specifically, we construct a framework in the unique non-acyclic affine case, the cyclically oriented $n$-cycle.
In the acyclic affine case, a framework was constructed by combining a copy of the Cambrian fan for $B$ with an antipodal copy of the Cambrian fan for~$-B$.
In this paper, we extend this ``doubled Cambrian fan'' construction to the oriented $n$-cycle, using a more general notion of sortable elements for quivers with cycles. 
\end{abstract}

\maketitle

\section{Introduction}
The close connection between exchange matrices and Cartan matrices leads naturally to the construction of combinatorial models for cluster algebras in terms of the root systems and Coxeter groups.
Such models have appeared, for example, in \cite{ga,ca2,camb_fan,framework}.  
This paper completes a project, begun in \cite{framework,afframe}, to construct explicit combinatorial models for the principal coefficients cluster algebra $\A_\bullet(B)$ for every $B$ whose Cartan companion $\Cart(B)$ of finite or affine type.
The key ingredients throughout the project are; (1) the notion of a \newword{framework} for $B$; and (2) the combinatorics and geometry of \newword{sortable elements}.
These two ideas are enough, in~\cite{framework}, to produce combinatorial models, called \newword{Cambrian frameworks}, for $\A_\bullet(B)$ when $\Cart(B)$ is of finite type.
When $\Cart(B)$ is of affine type, another idea is needed:  (3) doubling the Cambrian framework.
This third idea was applied, in~\cite{afframe}, to produce a \newword{doubled Cambrian framework} that serves as a combinatorial model for \emph{acyclic} exchange matrices $B$ such that $\Cart(B)$ is of affine type.
There is one infinite family of \emph{non-acyclic} exchange matrices $B$ such that $\Cart(B)$ is of affine type:  the oriented $n$-cycles.
In this paper, we use the notion of sortable elements for quivers with cycles~\cite{cyclic} and a variant of the doubling idea to construct a combinatorial model for $\A_\bullet(B)$ when $B$ is an oriented cycle.

For the remainder of this introduction, we briefly elaborate on these three ideas, with the goal of putting the present work into context, and then state the main results of this paper and some of their consequences for cluster algebras.

The \newword{Cartan companion} of $B$ is the matrix $A=\Cart(B)$ with entries $a_{ij} = -|b_{ij}|$ for $i \neq j$, and $a_{ii} = 2$.
The information contained in $B$ is equivalent to the Cartan matrix $\Cart(B)$ and an orientation $\Omega$ of the diagram of $\Cart(B)$. 
The orientation $\Omega$ is obtained by directing $i \from j$ if $b_{ij}>0$.
We say that $B$ is \newword{acyclic} if $\Omega$ has no directed cycles.

A framework
 is a labeled graph, which, under certain conditions, is isomorphic to the exchange graph of $\A_\bullet(B)$ and from which one can read off combinatorial data such as the $\c$-vectors, exchange matrices, $\g$-vector fan and exchange graph. 
Specifically, each vertex $v$ of the framework is labeled with a set of vectors that encode the $\c$-vectors.
The labels on $v$ define a cone $\Cone(v)$ whose rays encode the $\g$-vectors.
One of the conditions required for a framework to model the exchange graph and $\g$-vector fan is that it be complete.
A framework that is not complete models only part of the exchange graph and $\g$-vector fan.

In~\cite{framework}, we built on work in \cite{sortable,camb_fan,typefree} to construct a Cambrian framework for any acyclic $B$. 
Given a Coxeter group $W$, choosing an acyclic orientation of the Coxeter diagram of $W$ is equivalent to choosing a Coxeter element $c$ in $W$.
The $c$-sortable elements are certain elements of $W$, characterized in terms of the combinatorics of reduced words, or in terms of inversion sets.  
The $c$-Cambrian framework $\Camb_c$, as a graph, is the Hasse diagram of the restriction of the weak order to $c$-sortable elements.  
The labels in the $c$-Cambrian framework are certain roots that can be read off from the canonical reduced words ($c$-sorting words) for $c$-sortable elements.
For each $c$-sortable element $v$, we write $\Cone_c(v)$ rather than $\Cone(v)$ to emphasize the dependence on $c$.
The fan whose maximal cones are $\Cone_c(v)$ for all vertices $v$ of the $c$-Cambrian framework is called the $c$-Cambrian fan.

The $c$-Cambrian framework is complete when $W$ is a finite Coxeter group, but otherwise it is not complete.
The obstruction to completeness is that, for a sortable element $v$, the cone $\Cone_c(v)$ contains the cone $vD$ (the image of the fundamental chamber $D$ under the element $v$).
In particular, each cone $\Cone_c(v)$ meets the interior of the Tits cone $\bigcup_{w\in W}wD$.
However, in infinite type, the $\g$-vector fan has cones that don't meet the interior of the Tits cone. 
Thus we cannot expect the Cambrian construction to describe the entire $\g$-vector fan when $\Cart(B)$ is of infinite type.

We now describe a strategy, introduced in~\cite{afframe}, to escape the Tits cone when $\Cart(B)$ is of affine type. 
In finite type, the $c$-Cambrian fan and the $c^{-1}$-Cambrian fan are related by the antipodal map.
(See \cite[Proposition~1.3]{sort_camb}, \cite[Corollary~3.25]{afframe}, \cite[Remark~3.26]{afframe}, and \cite[Proposition~7.1]{universal}.) 
In affine type, this suggests the following strategy:
Take the graphs $\Camb_c$ and $\Camb_{c^{-1}}$. 
If $v$ is a vertex of $\Camb_c$ and $v'$ is a vertex of $\Camb_{c^{-1}}$ such that $\Cone_c(v) = - \Cone_{c^{-1}}(v')$, then glue $v$ to $v'$, to form a graph $\DCamb_c$. 
Since $\Cone_c(v)=-\Cone_{c^{-1}}(v')$ if and only if the labels on $v'$ are the negatives of the labels on $v$, the combined graph can be labeled consistently to form a framework and an associated fan, the doubled $c$-Cambrian framework and doubled $c$-Cambrian fan.
The doubled $c$-Cambrian fan covers all the points of the Tits cone and the antipodal Tits cone.
When $A$ is of affine type, the Tits cone is essentially a halfspace,
so the doubled $c$-Cambrian fan covers all of space except for part of the boundary of the Tits cone.
We obtain, when $B$ is acyclic and $\Cart(B)$ is of affine type, a complete framework for $B$.

Most affine Coxeter diagrams are trees, so all of their orientations are acyclic. 
The exception is the Coxeter group $\tilde{A}_{n-1}$, whose Coxeter diagram is an $n$-cycle 
and can thus be oriented cyclically.
In this paper, we analyze the case of the cyclically oriented $n$-cycle. 
In~\cite{cyclic}, we explained how to define sortable elements for any orientation of a Coxeter diagram, with cycles or not, and the basic idea for the oriented cycle is the same as for the acyclic affine case: glue together two antipodal combinatorial structures, each built on sortable elements.

The lack of acyclicity complicates matters, however.
In the non-acyclic case, the combinatorics of sortable elements does not provide enough labels to define simplicial cones $\Cone(v)$.
Thus many of the cones $\Cone(v)$ are defined by fewer than $n$ inequalities (in an $n$-dimensional vector space).
In particular, the cones do not combine to form a fan, and there is no Cambrian framework covering the Tits cone.
We solve this problem by taking another view of how to double the Cambrian fan.
In the affine acyclic case, we can describe the glueing of vertices in a different manner: 
A $c$-sortable $v$ in $\Camb_c$ is glued to a $c^{-1}$-sortable $v'$ in $\Camb_{c^{-1}}$ if and only if $\Cone_c(v) \cap \partial \Tits(A)= - \Cone_{c^{-1}}(v') \cap \partial \Tits(A)$.
Applying the same gluing criterion for the oriented cycle, several surprising things happen: 
the combined vertices all have exactly $n$ labels (the combined label sets of the two identified vertices);
the glued graph with combined label sets forms a framework; and
the labels define a complete simplicial fan (which is therefore the $\g$-vector fan).

There is
 an alternate approach to the cluster algebra defined by the oriented $n$-cycle.
This cluster algebra is of finite type $D_n$ so, if one is willing to change the base seed, one can study the cluster algebra of the cyclically oriented $n$-cycle using type $D$ combinatorics.
In this paper we will not do this. Rather, we work with the combinatorics of the affine Coxeter group $\tilde{A}_{n-1}$, thus preserving the rotational symmetry and making the oriented $n$-cycle our base seed.

We conclude this introduction with a more detailed summary of the main results and their consequences.
Let $\Omega$ be the cyclically oriented $n$-cycle, considered as a quiver.
Let $B$ be the corresponding $B$-matrix.
So $\Cart(B)$ is of affine type $\tilde A_{n-1}$.
We will also refer to $B$ itself as the oriented $n$-cycle.

We construct a regular graph $\DCamb_\Omega$ and a labeling $DC_\Omega$ such that:  
\begin{theorem}\label{cycle thm}
$(\DCamb_\Omega, \DC_\Omega)$ is a complete, descending, exact, well-connected, polyhedral, simply connected framework.  
\end{theorem}

The main significance of Theorem~\ref{cycle thm} is in its consequences for the combinatorics of the principal-coefficients cluster algebra $\A_\bullet(B)$ associated to $B$.
In particular, the graph $\DCamb_\Omega$ is isomorphic to the exchange graph of $\A_\bullet(B)$, and the exchange matrices, $\c$-vectors, $\g$-vectors, and the $\g$-vector fan can be read off from $(\DCamb_\Omega, \DC_\Omega)$.  
We now explain these consequences in greater detail, assuming the most basic definitions and facts about cluster algebras as found, for example, in \cite{ca4}.
For an exposition in the context of frameworks, see \cite[Section~3]{framework}.

We use the set $I(v)$ of edges incident to $v$ as an indexing set both for labels on $v$ and for the associated cluster variables.
The labeling $\DC_\Omega$ associates to each vertex $v$ of $\DCamb_{\Omega}$ a set $\DC_\Omega(v)=\set{\DC_\Omega(v,e):e\in I(v)}$ of vectors that form a basis for the root lattice. 
We write $\Cone_\Omega(v)$ for $\bigcap_{\beta\in \DC_\Omega(v)}\set{x\in V^*:\br{x,\beta}\ge0}$.
We also write $R_\Omega(v)$ for the basis of the weight space that is dual to $\DC_\Omega(v)$ and write $R_\Omega(v,e)$ for the vector in $R_\Omega(v)$ that is dual to $\DC_\Omega(v)$.
The symbol $\omega$ stands for the bilinear form defined by $B$.
The following corollary of Theorem~\ref{cycle thm} follows from \cite[Theorem~3.25]{framework}, \cite[Theorem~3.26]{framework}, and \cite[Corollary 4.6]{framework}.

\begin{cor}\label{cycle exchange} 
There exists a graph isomorphism $v\mapsto\Seed(v)=(\tB^v,X^v)$ from $\DCamb_\Omega$ to the principal coefficients exchange graph $\Ex_\bullet(B)$, where $B$ is the oriented $n$-cycle.
If $v$ is a vertex of $\DCamb_\Omega$, then:
\begin{enumerate}
\item
The exchange matrix $B^v=[b^v_{ef}]_{e,f\in I(v)}$ associated to $\Seed(v)$ is given by \[{b^v_{ef}=\omega(DC_\Omega(v,e),DC_\Omega(v,f))}.\]
\item
For each $e\in I(v$), the $\c$-vector $\c_e^v$ is $C(v,e)$. 
\item
If $X^v=(x^v_e:e\in I(v))$ is the cluster in $\Seed(v)$, then for each $e\in I(v$), the $\g$-vector $\g^v_e=\g(x^v_e)$ is $R(v,e)$.
\end{enumerate} 
Furthermore, the $\g$-vector cone of $\Seed(v)$ is $\Cone_\Omega(v)$, so that the map $v\mapsto\Cone_\Omega(v)$ is an isomorphism from $\DCamb_\Omega$ to the dual graph of the fan defined by $\g$-vectors of clusters in $\A_\bullet(B)$.
\end{cor}

Many of the standard conjectures for cluster algebras are collected in \cite[Section~3.3]{framework}.
These include, among others, the following assertions: the sign-coherence and basis properties hold for $\g$-vectors; the $\g$-vectors define a fan; each $\c$-vector has a sign; and a seed is determined by its $\c$-vectors.
Most of the standard conjectures are now theorems in light of \cite{GHKK}.
The construction of the framework $(\DCamb_\Omega, \DC_\Omega)$ establishes them in the case of the oriented cycle.

\begin{cor}\label{all conj cycle}
If $B$ is the oriented $n$-cycle, then Conjectures~3.9--3.20 from \cite{framework} all hold for $B$. 
\end{cor}

Corollary~\ref{all conj cycle} follows from Theorem~\ref{cycle thm}, using Theorems~\cite[Theorem~4.1]{framework}, \cite[Theorem~4.2]{framework}, \cite[Corollary 4.4]{framework}, and \cite[Corollary 4.6]{framework}.
In the case of \cite[Conjecture~3.19]{framework}, we must replace $B$ by $-B$ in our construction; see the paragraph after the statement of \cite[Conjecture~3.19]{framework}.

Sections~\ref{root system sec} through~\ref{cycle defn sec}  provide relevant background on Coxeter groups, frameworks and Cambrian combinatorics.
Section~\ref{main construction sec} introduces our main construction.
Sections~\ref{sec pairing} and~\ref{long computation} provide the combinatorial lemmas which prove that 
our main construction is well defined.
Section~\ref{proofs sec} proves that our construction has the promised properties.

\section{Root systems and Coxeter groups} \label{root system sec}
In this section, we fix our notation surrounding root systems and Coxeter groups.
We assume a basic familiarity with the topic.
Standard references include \cite{BjBr,Bour,Humphreys,Kac}.
For more background from the same point of view as this paper, see \cite{framework,afframe}.

Let $A$ be an
 $n \times n$ symmetric generalized Cartan matrix, meaning that $a_{ii}=2$, that $a_{ij} \leq 0$ for $i \neq j$, and that $a_{ij} = a_{ji}$. 
 We will write $I$ for an $n$-element set that indexes the rows and columns of $A$. 
When $A$ is of type $\tilde{A}_{n-1}$, we take $I$ to be the cyclic group $\ZZ/n \ZZ$, so objects indexed by $I$ have indices which are cyclic modulo $n$.
For example, in this case, $a_{i(i+1)} = a_{(i+1)i} = -1$ and $a_{ii}=2$.

All of the objects considered in this paper (cluster algebras, root systems, Coxeter groups and frameworks) work more generally with a symmetrizable Cartan matrix rather than a symmetric one.
Our goal in this paper is to study the case where $B$ corresponds to the oriented $n$-cycle, in which case $A$ is symmetric.
Accordingly, we will simplify our notation by describing the case where $A$ is symmetric and (later) where $B$ is skew-symmetric. 
Roots and co-roots are
the same when $A$ is symmetric, 
so our notation will not distinguish them. 

Let $V$ be a real $n$-dimensional vector space with a basis $\Pi=\set{\alpha_i:i\in I}$, let $V^*$ be the dual vector space, and write $\br{x,y}$ for the pairing between $x\in V^*$ and $y\in V$.
The vectors $\alpha_i$ are called the \newword{simple roots}.
The lattice spanned by the simple roots is called the \newword{root lattice}.
When $A$ is of type $\tilde{A}_{n-1}$, the positive roots are the vectors of the form $\alpha_i + \alpha_{i+1} + \cdots + \alpha_{j-1}$ with $i \leq j$ and $i \not \equiv j \bmod n$, indices interpreted modulo $n$. For example, if $n=3$, then one of the positive roots is $\alpha_1+\alpha_2+\alpha_3+\alpha_4 + \alpha_5 = 2 \alpha_1+2 \alpha_2 + \alpha_3$.

The Cartan matrix also determines a symmetric bilinear form $K$ on $V$ given by ${K(\alpha_i, \alpha_j)=a_{ij}}$.
As usual, we define the simple reflection $s_i$ to be the reflection orthogonal (with respect to $K$) to $\alpha_i$, so that $s_i(\alpha_j)=\alpha_j-a_{ij} \alpha_i$.
We write $W$ for the Coxeter group generated by the set $S$ of simple reflections.
The action of $W$ preserves the form $K$.
There is also a dual action of $W$ on $V^*$ as usual.

We write $\Phi$ for the (real) root system $\bigcup_{w \in W} w \cdot \{ \alpha_1, \ldots, \alpha_n \}$.
Each root $\beta$ defines a reflection $t$ by $tx=x-K(\beta, x) \beta$ for all $x\in V$.
We write $\beta_t$ for the \emph{positive} root associated to a reflection $t$.

We write $D$ for the \newword{fundamental chamber} $\{ x \in V^{\ast} : \langle \alpha_i, x \rangle  \geq 0 \}$ and write $\Tits(A)$ for the Tits cone $\bigcup_{w \in W} w D$. 
In this paper, we are concerned with $A$ of affine type $\tilde{A}_{n-1}$.
In this case, $\Tits(A)$ is $\{ x \in V^{\ast} : \langle \delta, x \rangle \, > 0 \} \cup \{ 0 \}$, where $\delta$ is the imaginary root $\alpha_1+ \alpha_2 + \cdots + \alpha_n$.
The kernel of $K$ is $\RR\delta$, so in particular, $\delta$ is fixed by the action of $W$.
The set of real roots $\Phi$ is invariant under translation by~$\delta$, and the image of $\Phi$ in $V/\RR \delta$ is a root system of finite type $A_{n-1}$ (that is to say, the root system for the Coxeter group $S_n$).

Given a Coxeter group $W$ with simple generators $S$, we use the term \newword{weak order} to refer to what is often called the \emph{right} weak order, with cover relations $w\covered ws$ when $\ell(w)<\ell(ws)$.

For each subset $J\subseteq S$, we write $W_J$ of $W$ for the \newword{standard parabolic subgroup} generated by $J$.
For each $w\in W$ and $J\subseteq S$, we write $w_J$ for the \newword{projection} of $w$ to $W_J$, the unique maximal element of $W_J$ below $w$ in the weak order.
We will often take $J$ to be $S\setminus\set{s}$, which we denote by $\br{s}$.

Let $w$ be an element of $W$.
An \newword{inversion} of $w$ is a reflection $t$ such that $\ell(tw)<\ell(w)$. 
We write $\inv(w)$ for the set of inversions ot $w$.
For any $J\subseteq S$, we have $\inv(w_J)=\inv(w)\cap W_J$.
A \newword{cover reflection} of $w$ is $t\in\inv(w)$ such that $tw=ws$ for some $s\in S$.
In this case, 
 $tw$ is covered by $w$ in the weak order.
We write $\cov(w)$ for the set of cover reflections of~$w$.
When $t$ is a cover reflection of~$w$, we have $\inv(tw)=\inv(w)\setminus\set{t}$.

\section{Frameworks for skew-symmetric exchange matrices.} \label{frameworks sec}
In this section, we give definitions and basic results on frameworks, in the special case where the initial exchange matrix is skew-symmetric.
For a discussion of the general skew-symmetrizable case, see \cite[Section~2]{framework}.

The underlying structure of a framework is a simple regular 
\newword{quasi-graph}.
This is a structure obtained from a simple graph by placing some number of \newword{half-edges} on each vertex.
Half-edges ``dangle'' from a vertex without connecting that vertex to any other.
Ordinary edges, i.e.\ edges connecting two vertices, will be called \newword{full edges}.
A quasi-graph is \newword{regular of degree $n$} if each vertex is incident to exactly $n$-edges.
(That is, each vertex $v$ has $k$ half-edges and $n-k$ full edges for some $k\in\set{0,\ldots,n}$ that depends on $v$.)
A quasi-graph is \newword{connected} if, when half-edges are ignored, it is connected in the usual sense.
An \newword{incident pair} is a pair $(v,e)$ where $v$ is a vertex $v$ and $e$ is an edge containing $v$.
We write $I(v)$ for the set of edges $e$ containing a vertex~$v$.

A skew-symmetric
 \newword{exchange matrix} $B=[b_{ij}]$ is a skew-symmetric integer matrix with rows and columns indexed by a set~$I$ with $|I|=n$.
The exchange matrix $B$ has a \newword{Cartan companion} $A=\Cart(B)=[a_{ij}]$ with diagonal entries~$2$ and off-diagonal entries $a_{ij}=-|b_{ij}|$.  
We define $\omega$ to be the skew-symmetric bilinear form with $\omega(\alpha_{i},\alpha_{j})=b_{ij}$ for all $i,j\in I$.

A \newword{framework} for an $n\times n$ exchange matrix $B$ is a pair $(G,C)$ such that $G$ is a connected, $n$-regular quasi-graph and $C$ is a function called the \newword{labeling} from incident pairs $(v,e)$ in $G$ to the vector space $V$, 
 satisfying certain conditions, called the
Sign,
Base, and  
Transition
conditions.
We write $C(v)$ for $\set{C(v,e):e\in I(v)}$.  

In the general setting of \cite{framework}, a framework is a triple $(G,C,C\ck)$, where $C$ and $C\ck$ are two labelings of $G$, one by vectors in the root lattice and one by vectors in the co-root lattice. 
In the skew-symmetric case, the two labelings coincide \cite[Remark 2.2]{framework}. 
Thus for skew-symmetric $B$, a framework is $(G,C,C)$, and we abbreviate this to $(G,C)$ in this paper.\\

\noindent 
\textbf{Sign condition:} \label{Sign condition} 
For each incident pair $(v,e)$, the label $C(v,e)$ is a nonzero vector, and either $C(v,e)$ or $-C(v,e)$ is in the nonnegative span of the simple roots.\\

For a nonzero vector $x$, we write $\sgn(x)=1$ if $x$ is in the nonnegative span of the simple roots and $\sgn(x)=-1$ if $-x$ is in the nonnegative span of the simple roots.
A natural way to satisfy the Sign condition is if all the $C(v,e)$'s are roots, but in general, the labels need not be roots. 
However, in the framework we will construct for the cyclically oriented $n$-cycle, all labels will be roots.
In \cite[Section~2.2]{framework}, we defined a notion of reflection framework; the reflection condition implies that all labels are roots.
However, our framework for the cyclically oriented $n$-cycle will not be a reflection framework.\\

\noindent
\textbf{Base condition:}  \label{Base condition}
There exists a vertex $v_b$ such that $C(v_b)$ is the set of simple roots. \\

We identify the indexing set $I$ with $I(v_b)$ by identifying $e\in I(v_b)$ with the index $i\in I$ such that $C(v_b,e)=\alpha_i$.
\\  

\noindent
\textbf{Transition condition:} \label{Transition condition}
Suppose $v$ and $v'$ are distinct vertices incident to the same edge $e$.
Then $C(v,e)=-C(v',e)$.
Furthermore, if $\beta=C(v,e)$
and $\gamma\in C(v)\setminus\set{\beta}$, then $\gamma+[\sgn(\beta)\omega(\beta,\gamma)]_+\,\beta$ is in $C(v')$.\\

One can easily check that the transition condition is symmetric in $v$ and $w$.

Given an edge $e$ connecting $v$ to $v'$, there is a bijection $\mu_e:I(v)\to I(v')$ with $\mu_e(e)=e$, defined to make the following condition equivalent to the Transition condition. 
 See \cite[Section 2.1]{framework}.\\

\noindent
\textbf{Transition condition, restated:} \label{Transition condition, strengthened}
Suppose $v$ and $v'$ are distinct vertices incident to the same edge $e$.
Then $C(v,e)=-C(v',e)$.
Furthermore, if $f\in I(v)\setminus\set{e}$, then 
\[C(v',\mu_e(f))=C(v,f)+[\sgn(C(v,e))\omega(C(v,e),C(v,f))]_+\,C(v,e).\]
 
We warn the reader that the inverse function $I(v') \to I(v)$ is also denoted $\mu_e$.

For each vertex $v$ in a framework $(G,C)$, define $\Cone(v)$ to be the cone in $V^{\ast}$ given by $\bigcap_{\beta\in C(v)}\set{x\in V^*:\br{x,\beta}\ge0}$.
By \cite[Proposition~2.4]{framework}, each $\Cone(v)$ is a simplicial cone.

In \cite{framework}, several global conditions on frameworks are defined, and the consequences of these conditions are discussed.
Here, we wish to bypass precise definitions of some of these conditions.
We can do this because we will establish one global condition, the descending condition, which implies many of the others.
For a full discussion, see \cite[Section~4]{framework}.

The framework $(G,C)$ is \newword{injective} if the map $v\mapsto C(v)$ 
is injective, and it is \newword{complete} if $G$ has no half-edges.
The framework is \newword{polyhedral} if the cones $\Cone(v)$ for $v\in G$ are the maximal cones of a fan and if the map $v\mapsto\Cone(v)$ is injective. 
The adjectives \newword{well-connected}, \newword{simply connected}, \newword{ample}, and \newword{exact} can also be applied to a framework as explained in \cite[Section~4]{framework}.

A framework is \newword{descending} if it satisfies the following three conditions.\\

\noindent
\textbf{Unique minimum condition:}   \label{Unique minimum condition}
If a vertex $v$ of $G$ has $\set{\sgn(\beta):\beta\in C(v)}=\set{1}$, then~$v$ is the base vertex $v_b$.\\

\noindent
\textbf{Full edge condition:}  \label{Full edge condition}
If $(v,e)$ is an incident pair and $\sgn(C(v,e)) = -1$ then $e$ is a full edge.\\

In light of the Sign and Transition conditions, we can orient the edges of $G$ as follows:
If $e$ is an edge incident to a vertex $v$, then we direct $e$ towards $v$ if $\sgn(C(v,e)) = 1$ and away from $v$ if $\sgn(C(v,e))=-1$.\\

\noindent
\textbf{Descending chain condition:}  \label{Descending chain condition}
There exists no infinite sequence $v_0\to v_1\to\cdots$.\\

For us, the descending property is the key to proving most of the other key properties of the framework $(\DCamb_\Omega, \DC_\Omega)$ for the oriented $n$-cycle.
This is because \cite[Theorem~4.20]{framework} says that a descending framework is exact, polyhedral, well-connected, and simply connected.
This reduces the proof of Theorem~\ref{cycle thm} to the following assertion:

\begin{theorem} \label{main theorem}
$(\DCamb_\Omega, \DC_\Omega)$ is a complete descending framework.
\end{theorem}

The rest of the paper is devoted to constructing $(\DCamb_\Omega, \DC_\Omega)$ and proving Theorem~\ref{main theorem}.

\section{Cambrian frameworks} \label{sec cambrian}  
In this section, we review  the construction of a Cambrian framework for any acyclic exchange matrix $B$. (See~\cite{framework}.)
The information contained in an acyclic exchange matrix $B$ is equivalent to $(A,c)$, where $A=\Cart(B)$ and $c$ is the Coxeter element obtained by ordering $S$ such that $s_i$ precedes $s_j$ whenever $b_{ij}>0$.
We will use the data $(A,c)$ to define a framework for $B$.
Since the exchange matrix $B$ of interest in this paper is skew-symmetric and thus $A$ is symmetric, we present the Cambrian framework construction in the special case of symmetric $A$.

As before, let $\Phi$ be the root system associated to $A$.
Let $W$ be the Coxeter group associated to $A$ and let $s_1 s_2 \cdots s_n$ be a reduced word for $c$.
We write $c^{\infty}$ for the infinite word $(s_1 s_2 \cdots s_n) (s_1 s_2 \cdots s_n) \cdots$ in the simple generators $s_i$. 
Let $w$ be an element of $W$.  The \newword{$c$-sorting word} for $w$ is the lexicographically leftmost subword of $c^{\infty}$ which is a reduced word for $w$.   
Let $J_r \subseteq S$ be the letters from the $r$-th copy of $(s_1 s_2 \cdots s_n)$ that occur in the $c$-sorting word for $w$.
The group element $w$ is \newword{$c$-sortable} if and only if $J_{1} \supseteq J_2 \supseteq J_3 \supseteq \cdots$. (See 
\cite[Section 2]{sortable}.)
The subword of $s_1\cdots s_n$ consisting of the letters indexed by $J_r$ is called the \newword{$r$-th block of the $c$-sorting word for $w$}.

Given $w\in W$, there is a unique maximum element, denoted $\pidown^c(w)$, among all $c$-sortable elements below $w$ in the weak order \cite[Corollary~6.2]{typefree}.   
The restriction of the weak order to $c$-sortable elements is a meet semilattice called the \newword{$c$-Cambrian semilattice}, and $\pidown^c$ is an order preserving map from the weak order to the $c$-Cambrian semilattice \cite[Theorem~6.1]{typefree}.

We will define a set of labels $C^r(v)$, indexed by the simple reflections, for each $c$-sortable element $v$.
See \cite[Section 5]{typefree} for more information, including an alternate recursive definition.

Let $s_{i_1} s_{i_2} \cdots s_{i_N}$ be the $c$-sorting word for $v$. 
Viewing this as a subword of $c^\infty$, there is some first occurrence of $r$ in $c^\infty$ which is not in $s_{i_1} s_{i_2} \cdots s_{i_N}$.
Suppose this occurrence falls between $s_{i_t}$ and $s_{i_{t+1}}$. 
Then we say that $s_{i_1} s_{i_2} \cdots s_{i_N}$ \newword{skips} $r$ between $s_{i_t}$ and $s_{i_{t+1}}$.
(If $t=0$, we mean that the skip is before $s_1$ and if $t=N$, we mean that the skip is after $s_N$.)
We define $C_c^r(v) = s_{i_1} s_{i_2} \cdots s_{i_t} \alpha_r$ and $C_c(v)=\set{C_c^r(v):r\in S}$.

To realize $C_c$ as a labeling on incident pairs, we use the following lemma, which is part of \cite[Lemma~5.11]{framework}.  
\begin{lemma}\label{cov beta part 1}
If $v'\covered v$ in the $c$-Cambrian semilattice, then there exists a unique root $\beta$ such that $\beta\in C_c(v')$ and $-\beta\in C_c(v)$.
The root $\beta$ is positive.
\end{lemma}
For each edge $e$ corresponding to a cover relation $v'\covered v$ in $\Camb_c$, we define $C_c(v,e)=-\beta$ and $C_c(v',e)=\beta$.
The degree of each vertex of $\Camb_c$ is at most $n$.
If the degree is less than $n$, then we attach half-edges to $v$ to make the degree $n$.
These half-edges are labeled with the labels in $C_c$ that were not already associated to full edges.
We overload the notation $\Camb_c$ to also denote the quasi-graph with these added half-edges, and also re-use the symbol $C_c$ for the labeling of incident pairs in $\Camb_c$.
The following is \cite[Theorem~5.12]{framework}.

\begin{theorem}\label{camb frame}  
$(\Camb_c,C_c)$ is a descending framework for the exchange matrix~$B$.
\end{theorem}
We call $(\Camb_c,C_c)$ the \newword{Cambrian framework} for $B$.
In fact, \cite[Theorem~5.12]{framework} stated that ``$(\Camb_c,C_c)$ is a descending reflection framework, '' but \cite[Proposition~2.13]{framework} allows us to restate the theorem as above.
For more on reflection frameworks and how they relate to frameworks, see \cite[Section~2.2]{framework}.

For any vertex $v$ of a framework, we defined the corresponding cone $\Cone(v)$, dual to the vectors $C(v)$. 
In the context of a Cambrian framework, we write $\Cone_c(v)$ for clarity.
The cones $\Cone_c(v)$ and their faces form a fan~\cite[Corollary 5.15]{framework}.  
This fan is denoted $\F_c$ and called the \newword{$c$-Cambrian fan}.
The polyhedral geometry of the cones~$\Cone(v)$ is related to the combinatorial operation $\pidown^c$. 
The following theorem is \cite[Theorem~6.3]{typefree}.
(The second assertion of the theorem doesn't appear in \cite[Theorem~6.3]{typefree}, but is an immediate consequence of the first, since $\pidown^c(v)=v$.)

\begin{theorem} \label{pidown fibers}
Let $v$ be $c$-sortable. Then $\pidown^c(w)=v$ if and only if $w D \subseteq \Cone_c(v)$.
In particular, $vD$ is contained in $\Cone(v)$.
\end{theorem}

The following proposition is a special case of \cite[Corollary~3.25]{afframe}, as explained in \cite[Remark~3.26]{afframe}.
Alternately, the proposition is stated in terms of Cambrian (lattice) congruences as \cite[Proposition~1.3]{sort_camb}.

\begin{prop} \label{double finite} 
When $W$ is finite, the fans $\F_c$ and $-\F_{c^{-1}}$ are identical.
Specifically, for any $c$-sortable element $v$, we have $\Cone_c(v)=-\Cone_{c^{-1}}(v')$ where $v' = \pidown^{c^{-1}}(v w_0)$.
\end{prop}

Finally, we gather some properties of the labeling $C_c$.
Recall that $\br{r}$ is~${S\setminus\set{r}}$.  
\begin{lemma} \label{sign omit}
If $v$ is in the parabolic subgroup $W_{\br{r}}$, then $C^r_c(v) \in \Phi^{+}$.
\end{lemma}

\begin{proof}   
Suppose the $c$-sorting word $s_{i_1} s_{i_2} \cdots s_{i_N}$ for $v$ skips $r$ between $s_{i_t}$ and $s_{i_{t+1}}$.
That is, the first occurrence of $r$ in $c^\infty$ that is not in $s_{i_1} s_{i_2} \cdots s_{i_N}$ occurs between $s_{i_t}$ and $s_{i_{t+1}}$.
Then the root $C_c^r(v)$ is positive if and only if $s_{i_1} s_{i_2} \cdots s_{i_t} r$ is reduced.
If $v$ is in $W_\br{r}$, then each simple reflection occurs at most once in the word $s_{i_1} s_{i_2} \cdots s_{i_t} r$, so this word is reduced.
\end{proof}

Given a Coxeter element $c$ and $J\subseteq S$, the \newword{restriction} of $c$ to $W_J$ is the Coxeter element of $W_J$ obtained by deleting the letters in $S\setminus J$ from a reduced word for $c$. 
Typically the restriction of $c$ to $W_J$ does not equal the projection $c_J$.
The first two of the following propositions are \cite[Proposition~5.2]{typefree} and \cite[Proposition~2.30]{typefree}.
Recall that $\cov(v)$ denoted the set of cover reflections of $v$, as defined in Section~\ref{root system sec}.
\begin{prop}\label{lower walls}
Let $v$ be a $c$-sortable element.
The set of negative roots in $C_c(v)$ is $\set{-\beta_t:t\in\cov(v)}$.
\end{prop}

\begin{prop}\label{sort para easy}
Let $J \subset S$ and let $c'$ be the restriction of $c$ to $W_J$. 
Then an element $v \in W_J$ is $c$-sortable if and only if it is $c'$-sortable. 
\end{prop}

\begin{prop}\label{Cc para}
Let $J \subset S$ and let $c'$ be the restriction of $c$ to $W_J$. 
If $v \in W_J$ is $c$-sortable, then $C_{c'}^r(v)=C_c^r(v)$ for all $r\in J$.  
Thus ${C_{c'}(v)=\set{C_c^r(v):r\in J}\subseteq C_c(v)}$.
\end{prop}
\begin{proof}
Proposition~\ref{sort para easy} implies that the notation $C_{c'}(v)$ makes sense.
The $c'$-sorting word for $v$ coincides with the $c$-sorting word for $v$, and thus skips each $r\in J$ in the same place.
\end{proof}

We conclude this section with a lemma about $c$-sortable elements in type A. 
\begin{lemma}\label{easy lem}
Suppose $W$ is of type $A_k$, with vertices numbered in order along the Coxeter diagram, and let $c$ be the Coxeter element $s_1\cdots s_k$.
If $v\in W$ is $c$-sortable and has full support, then $v=cu$ where $u$ is a $c'$-sortable element in $W_\br{s_k}$ for $c'=s_1\cdots s_{k-1}$.
\end{lemma}
\begin{proof}
By definition of $c$-sortability, $v=cu$ for some $c$-sortable element $u$ of $W$ with $\ell(v) = \ell(c) + \ell(u)$.
So $u \leq c^{-1} w_0$. One computes that $c^{-1} w_0$ is the longest element in $W_{\br{s_k}}$, so $u \in W_{\br{s_k}}$. 
The result follows by Proposition~\ref{sort para easy}.
\end{proof}

\section{Cambrian combinatorics with cycles} \label{cycle defn sec}   
The construction of the Cambrian framework (and the ``doubled'' Cambrian framework described in the introduction) relied crucially on the hypothesis that $B$ is acyclic.
However, the object of study in this paper is a non-acyclic exchange matrix: the oriented cycle.
In~\cite{cyclic}, the notion of sortable elements was generalized to allow for non-acyclic orientations of the Coxeter diagram of $W$.
In this section, we review the main results of~\cite{cyclic}.

Let $B$ be an $n \times n$ skew-symmetrizable integer matrix,
 $W$ the corresponding Weyl group and $S$ the simple generators.  
Let $\Omega$ be the directed graph with vertex set $S$ and edges $s_i \from s_j$ whenever $b_{ij} > 0$.
For a subset $J$ of $S$, we say that $J$ is \newword{$\Omega$-acyclic} if the subgraph of $\Omega$ induced by $J$ is acyclic.
If $J$ is $\Omega$-acyclic, we write $c(\Omega, J)$ for the product of the elements $\{ s_i \}_{i \in J}$, ordered with $s_i$ before $s_j$ if there is an edge $s_i \from s_j$.  
This is a Coxeter element of $W_J$.

For $w \in W$, define $J(w)$
to be the unique minimal subset $J$ 
of $S$ so that $w$ is in the parabolic subgroup $W_J$.
We define $w$ to be \newword{$\Omega$-sortable} if $J(w)$ is $\Omega$-acyclic and $w$ is $c(\Omega, J)$-sortable.
In~\cite{cyclic}, we define an idempotent map $\pidown^{\Omega}$ from $W$ to the $\Omega$-sortable elements, mapping $w\in W$ to the unique maximal $\Omega$-sortable element below $w$. 
We restate the defining property of $\pidown^{\Omega}$ as a lemma:
\begin{lemma}\label{helpful}
If $x$ is $\Omega$-sortable and $y\in W$, then $x\le y$ if and only if $x\le\pidown^\Omega(y)$.
\end{lemma}

We also show that the $\Omega$-sortable elements and the map $\pidown^{\Omega}$ have the same lattice-theoretic properties as in the acyclic case.  
The \newword{$\Omega$-Cambrian semilattice} is the subposet of the weak order on $W$ induced on the $\Omega$-sortable elements.
We define the graph $\Camb_{\Omega}$ to be the Hasse diagram of the $\Omega$-Cambrian semilattice.

Let $v \in W$ be $\Omega$-sortable and let $r$ be a simple reflection of $W$. If $J(v) \cup \{ r \}$ is $\Omega$-acyclic, we define $C_{\Omega}^r(v)$ to be $C_{c(\Omega, J(v) \cup \{ r \} )}^r(v)$.  
If $J(v) \cup \{ r \}$ is not $\Omega$-acyclic, we leave $C_{\Omega}^r(v)$ undefined.  
Define $C_{\Omega}(v) = \{ C_{\Omega}^r(v) : \ J(v) \cup \{ r \}\ \mbox{is\ $\Omega$-acyclic} \}$. 
In particular, $C_{\Omega}(v)$ can have fewer than $n$ elements.
When $\Omega$ is acyclic, the set $C_{\Omega}(v)=\set{C_{\Omega}^r(v) :r\in S}$ has $n$ elements for every $v$, and furthermore, setting $c=c(\Omega,S)$, we have $C_\Omega(v)=C_c(v)$.
We will repeatedly use the following lemma: 
\begin{lemma} \label{observe}  
If $v$ is an $\Omega$-sortable element of $W$, and $v \in W_J$ for some $\Omega$-acycilc $J \subseteq S$, then $C_{c(\Omega,J)}(v) \subset \C_\Omega(v)$.
\end{lemma}

\begin{proof} 
For any $r\in J$, we have $J(v)\cup\set{r}\subseteq J$, so Proposition~\ref{Cc para} says that $C^{r}_{c(\Omega,J)}(v) = C^{r}_{c(\Omega,J(v) \cup \{ r \})}(v) \in C_{\Omega}(v)$. 
\end{proof}

If $e = (v', v)$ is an edge of the Hasse diagram $\Camb_\Omega$, 
corresponding to a cover $v' \covered v$, then we can apply Lemma~\ref{cov beta part 1} within standard parabolic subgroups to conclude that there exists a unique root $\beta$ such that $\beta\in C_\Omega(v')$ and $-\beta\in C_\Omega(v)$, and that $\beta$ is positive.
We thus label the incident pairs $(v,e)$ and $(v',e)$ by $\pm \beta$ as in the acyclic case. 
We can add half-edges to $\Camb_{\Omega}$ to make an $n$-regular quasi-graph as in the acyclic case, but some vertices do not have enough labels to put on the added half-edges.
Thus we cannot define a framework when $\Omega$ is acyclic without some additional information. 

Each $\Omega$-sortable element $v$ defines a cone 
\[\Cone_{\Omega}(v) := \{ x \in V^{\ast} : \langle x, \beta \rangle \geq 0,\ \beta \in C_{\Omega}(v) \}.\] 
When $\Omega$ is not acyclic, these cones are not the maximal cones of a fan; they do not even have the same lineality space.
However, it is still true \cite[Theorem~4.1]{cyclic} that, for any $\Omega$-sortable $v$ and any $w \in W$, we have $\pidown^{\Omega}(w) = v$ if and only if $wD \subseteq \Cone_{\Omega}(v)$.

\section{Constructing a framework for the oriented cycle} \label{main construction sec}
In this section, we give the construction of $(\DCamb_\Omega, \DC_\Omega)$.
The proof that it is a framework is given in Section~\ref{proofs sec}.
We begin by illustrating our construction for $n=3$.

\begin{example}
Consider the Coxeter group of type $\tilde{A}_2$.
The diagram of this Coxeter group is a triangle.
Let $\Omega$ be the orientation $s_1 \from s_2 \from s_3 \from s_1$.   
There are ten $\Omega$-sortable elements. The intersection of the cones $\Cone_{\Omega}(v)$ with
$V_1^*:=\set{x:\br{x,\delta}=1}$
are drawn with dark edges on the left side of Figure~\ref{cyclic1}. 
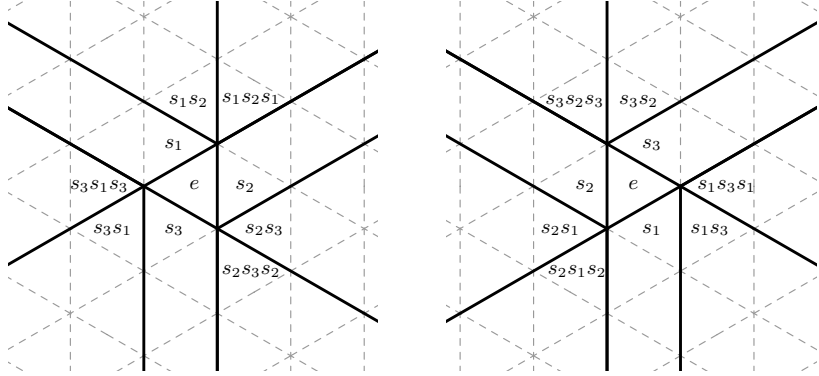
\begin{figure}[h] \begin{center}
\scalebox{.87}{\includegraphics{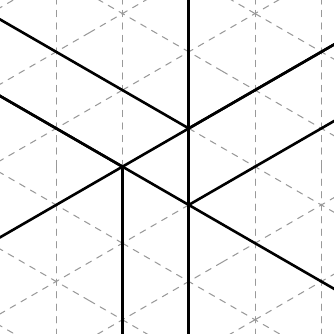}
\begin{picture}(0,0)(80,-80)
\put(-6,-1.5){\small$e$}
\put(-18,16.5){\small$s_1$}
\put(-16,36){\small$s_1s_2$}
\put(9,41){\small$s_1s_2s_1$}
\put(14,-1.5){\small$s_2$}
\put(18,-20.5){\small$s_2s_3$}
\put(9,-42){\small$s_2s_3s_2$}
\put(-18,-20.5){\small$s_3$}
\put(-47,-20.5){\small$s_3s_1$}
\put(-63,-1.5){\small$s_3s_1s_3$}
\end{picture}}
\qquad
\scalebox{.87}{\includegraphics{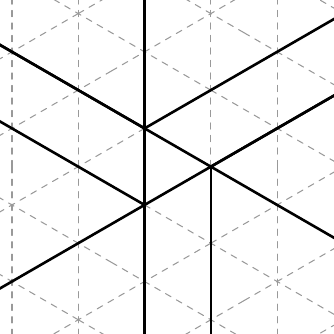}
\begin{picture}(0,0)(80,-80)
\put(-6,-1.5){\small$e$}
\put(1,17){\small$s_3$}
\put(-11,35.5){\small$s_3s_2$}
\put(-47,41){\small$s_3s_2s_3$}
\put(-30,-1.5){\small$s_2$}
\put(21,-20.5){\small$s_1s_3$}
\put(25,-1.5){\small$s_1s_3s_1$}
\put(1,-20.5){\small$s_1$}
\put(-44,-20.5){\small$s_2s_1$}
\put(-46,-45){\small$s_2s_1s_2$}
\end{picture}}
\end{center}
\caption{$\Omega$-sortable elements, $(-\Omega)$-sortable elements, and their cones}
\label{cyclic1}
\end{figure}
The gray dotted lines indicate the decomposition of the Tits cone into cones $wD$.
The boundary of $\Tits(A)$ is not represented in Figure~\ref{cyclic1}, but corresponds to the ``circle at infinity'' in each picture.
There are six $\Omega$-sortable elements for which $\Cone_{\Omega}(v) \cap \partial \Tits(A)$ is $(n-1)$-dimensional (i.e.\ $2$-dimensional).

The right picture of Figure~\ref{cyclic1} is the analogous, but antipodal, picture for the opposite orientation $-\Omega$.
Once again, there are six $(-\Omega)$-sortable elements for which $-\Cone_{-\Omega}(v) \cap \partial \Tits(A)$ is $(n-1)$-dimensional.
Moreover, these two sets of six three-dimensional cones intersect $\partial \Tits(A)$ in the same six two-dimensional cones.
Both pictures are drawn from the point of view of a point in $\Tits(A)$ lying on the other side of $V_1^*$ from the origin.

We glue the graphs $\Camb_\Omega$ and $\Camb_{-\Omega}$ together by identifying the $\Omega$-sortable element $v_1$ with the $(-\Omega)$-sortable element $v_2$ if $\Cone_{\Omega}(v_1) \cap \partial \Tits(A)={- \Cone_{- \Omega}(v_2) \cap \partial \Tits(A)}$. 
Writing $v$ for the resulting vertex, we set $\DC_\Omega(v) := C_{\Omega}(v_1) \cup - C_{-\Omega}(v_2)$.  
The corresponding operation in the dual space is to assign the cone $\Cone_{\Omega}(v_1) \cap - \Cone_{-\Omega}(v_2)$ to $v$.
These cones form the maximal cones of a fan $\DF_{\Omega}$, pictured in Figure~\ref{cyclic2}. We draw codimension-$1$ faces in the closure of $\Tits(A)$ in black, codimension-$1$ faces contained in the closure of $-\Tits(A)$ in blue, and codimension-$1$ faces that cross $\partial\Tits(A)$ in dotted blue and black. 
The red circle indicates $\partial \Tits$.
This will be the fan of our framework for the cyclically oriented triangle.
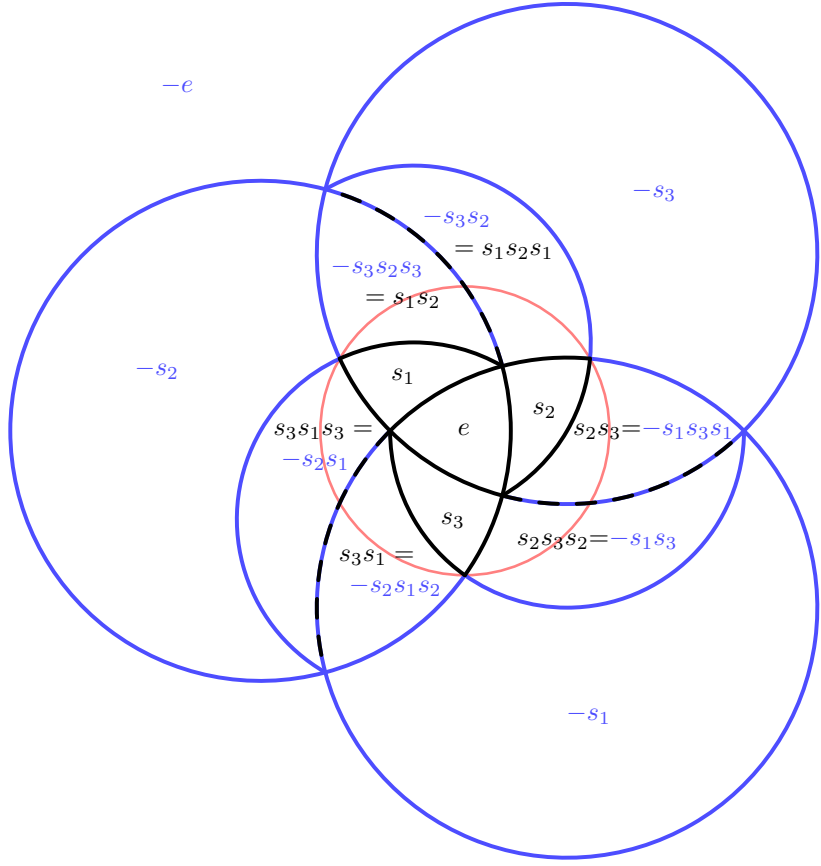
\begin{figure}[h] 
\begin{center}
\definecolor{myblue}{rgb}{0.3,0.3,1}
\begin{picture}(0,0)
\end{picture}
\scalebox{0.99}{\includegraphics{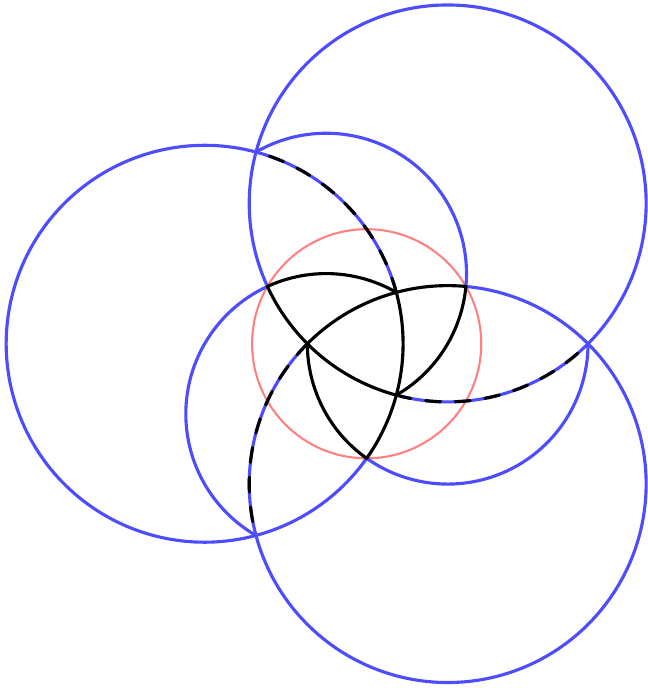}
\begin{picture}(0,0)(135,-165)
\put(-6.5,-1){$e$}
\put(-32,20){$s_1$}
\put(-56,61){\textcolor{myblue}{$-s_3s_2s_3$}}
\put(-42,49){$=s_1s_2$}
\put(-23,80){\textcolor{myblue}{$-s_3s_2$}}
\put(-11,68){$=s_1s_2s_1$}
\put(21,7){$s_2$}
\put(38,0){$s_2s_3$=}
\put(41,-12){\textcolor{myblue}{$-s_1s_3s_1$}}
\put(12,-43){$s_2s_3s_2$=\textcolor{myblue}{$-s_1s_3$}}
\put(-14,-34){$s_3$}
\put(-56,-49){$s_3s_1=$}
\put(-52,-61){\textcolor{myblue}{$-s_2s_1s_2$}}
\put(-81,-3){$s_3s_1s_3=$}
\put(-76,-14){\textcolor{myblue}{$-s_2s_1$}}
\put(35,-110){\textcolor{myblue}{$-s_1$}}
\put(-130,22){\textcolor{myblue}{$-s_2$}}
\put(60,90){\textcolor{myblue}{$-s_3$}}
\put(-120,130){\textcolor{myblue}{$-e$}}
\end{picture}}
\end{center}
\caption{The fan $\DF_\Omega$}
\label{cyclic2}
\end{figure}
\end{example}

\begin{remark}\label{Shi remark}  
If all walls in Figure~\ref{cyclic1} are extended to lines, the result is a drawing of the Shi arrangement of type $A_2$, first studied in \cite{Shi}.
Indeed, one easily deduces from the results of this paper that 
the roots appearing as $c$-vectors for the oriented $n$-cycle are exactly the roots of $\tilde A_{n-1}$ that define the Shi arrangement of type $A_{n-1}$.
\end{remark}

We now show how the above example generalizes to the cyclicly oriented $n$-cycle for $n>3$. From now on, $\Omega$ is the oriented graph $s_1 \from s_2 \from \cdots \from s_n \from s_1$ and $-\Omega$ is the reverse orientation.
Indices should always be treated as cyclic modulo $n$.

In Section~\ref{sec pairing},
we will prove the following proposition.
\begin{prop} \label{pairing summary}
For every $\Omega$-sortable $v$, the set of roots $C_{\Omega}(v)$ contains either $n$ or $n-1$ elements. There is a bijection 
\[ \eta : \{ v \in \Camb_{\Omega} : |C_{\Omega}(v)| = n-1 \} \to  \{ v' \in \Camb_{-\Omega} : |C_{-\Omega}(v')| = n-1 \}  \]
such that $C_{\Omega}(v) \cup - C_{- \Omega}(\eta(v))$ contains $n$ roots.
\end{prop}
There is a geometric description of this pairing: $|C_{\Omega}(v)| = n-1$ if and only if $\Cone_{\Omega}(v)$ has $n-1$ dimensional intersection with $\partial \Tits(A)$ and, in this case, $\Cone_{\Omega}(v) \cap \partial \Tits(A) = - \Cone_{-\Omega}(\eta(v))\cap \partial \Tits(A)$.

Let $\DCamb_\Omega$ be obtained from the disjoint union $\Camb_\Omega\cup(-\Camb_{-\Omega})$ by identifying each $v\in\Camb_\Omega$
having $|C_{\Omega}(v)| = n-1$ with $-\eta(v)\in -\Camb_{-\Omega}$.
We define a labeling $\DC_{\Omega}$ of $\DCamb_{\Omega}$ as follows: If $v$ is $\Omega$-sortable with $|C_{\Omega}(v)| = n$ then $\DC_{\Omega}(v) = C_{\Omega}(v)$. If $v'$ is $-\Omega$-sortable with $|C_{-\Omega}(v') | = n$, then $\DC_{\Omega}(v') = - C_{-\Omega}(v')$. 
If $v$ is $\Omega$-sortable with $|C_{\Omega}(v)| = n-1$, then $\DC_{\Omega}(v) = C_{\Omega}(v) \cup (- C_{-\Omega}(\eta(v)))$.
We have described $\DC_\Omega$ in terms of the label sets $\DC_\Omega(v)$.
It is straightforward now to work out labels on incident pairs.

In Sections~\ref{sec pairing} 
and~\ref{long computation}, we elaborate on and prove Proposition~\ref{pairing summary}.
Then we continue on to the proof of Theorem~\ref{main theorem}.

\section{The bijection $\eta$} \label{sec pairing}
In this section, we define the map $\eta$ of Proposition~\ref{pairing summary}.
We begin by giving a much more precise description  of when $C_{\Omega}(v)$ has $n-1$
or $n$ roots.
Recall the notation $\br{s}$ for $S\setminus\set{s}$.

\begin{prop} \label{TwoTypes}
Let $v$ be $\Omega$-sortable. Then either:
\begin{enumerate}
\item[(1a)] $|J(v)| \leq n-2$ and 
\item[(2a)] $|C_{\Omega}(v)| = n$ and
\item[(3a)] $\Cone_{\Omega}(v)$ does not cross $\partial \Tits(A)$ and
\item[(4a)] The dimension of $\Cone_{\Omega}(v) \cap \partial \Tits(A)$ is less than $n-1$ 
\end{enumerate}
or else
\begin{enumerate}
\item[(1b)] $J(v) = \br{s_i}$ for some $i\in[n]$ and 
\item[(2b)] $|C_{\Omega}(v)| = n-1$, with $C_{\Omega}^{s_i}(v)$ undefined for the same $i$ as in~\textup{(1b)}, and  
\item[(3b)] $\Cone_{\Omega}(v)$ crosses $\partial \Tits(A)$ and 
\item[(4b)] The dimension of $\Cone_{\Omega}(v) \cap \partial \Tits(A)$ is $n-1$.
\end{enumerate}
\end{prop}

\begin{proof} 
The equivalence of (1$\bullet$) and (2$\bullet$) is immediate in each case, as is that of~(3$\bullet$) and~(4$\bullet$). Also, in each case, exactly one of the paired conditions holds.
Thus it is enough to establish the implications $\text{(1a)}\implies\text{(3a)}$ and $\text{(1b)}\implies\text{(3b)}$.

$\mbox{(1a)} \implies \mbox{(3a)}$: For each $s_i \in S\setminus J(v)$, if $s_{i-1}\in J(v)$, then let $\lambda(i)$ be such that $\{ s_{\lambda(i)}, s_{\lambda(i)+1}, \ldots,  s_{i-1} \} \subseteq J(v)$ and $s_{\lambda(i)-1} \not \in J(v)$. 
Otherwise, let ${\lambda(i) = i}$.
Either way, condition~(1a) says that  $\lambda(i) \neq i+1$, and 
one computes $C^{s_i}_{\Omega}(v) = {s_{\lambda(i)} s_{\lambda(i)+1} \cdots s_{i-1} \alpha_i} = {\alpha_{\lambda(i)} + \alpha_{\lambda(i)+1} + \cdots + \alpha_i}$. 
Then one has
\[ \sum_{s_i \in S \setminus J(v)} C_{\Omega}^{s_i}(v) = \sum_{s_i \in S\setminus J(v)} \left( \alpha_{\lambda(i)} + \alpha_{\lambda(i)+1} + \cdots + \alpha_i  \right)= \sum_{j=1}^n \alpha_j = \delta. \]
Thus $\br{\delta,\, \cdot\, }$ 
is nonnegative on $\Cone_{\Omega}(v)$, so $\Cone_{\Omega}(v)$ does not cross $\partial \Tits(A)$.

$\mbox{(1b)} \implies \mbox{(3b)}$: The $n-1$ roots of $C_{\Omega}(v)$ all lie in the sub root-system associated to the parabolic subgroup $W_{\br{s_i}}$. 
Thus $\Cone_{\Omega}(v)$ contains the entire line $\bigcap_{j\neq i}\alpha_j^\perp$.
In particular, $\Cone_{\Omega}(v)$ lies on both sides of $\partial \Tits(A)$.
\end{proof}

We introduce the following notations:
Let $w \mapsto w^{+}$ be the automorphism  of $W$ with $s_i^+ = s_{i+1}$. 
We also use the superscript $^+$ for the corresponding permutation of $\Phi$ and, when $C$ is a set of roots, write $C^+$ for $\set{\beta^+:\beta\in C}$. 
The inverse of ${}^+$ is written ${}^-$.
 
For integers $i$ and $j$, let $c[i,j]$ denote the product $s_i s_{i+1} \cdots s_{j-1} s_j$ if $i \leq j$ and $s_i s_{i-1} \cdots s_{j+1} s_j$ if $i \geq j$, with indices taken modulo~$n$.

Write $U_\Omega(s_i)$ for the set of $\Omega$-sortable elements $v$ described in case~(b) of Proposition~\ref{TwoTypes}, and similarly $U_{-\Omega}(s_i)$.
The union $\bigcup_{i\in[n]}U_\Omega(s_i)$ contains exactly the vertices of $\Camb_{\Omega}$ which will glue to vertices in $\Camb_{-\Omega}$. 
We will now define the bijection $\eta:U_{\Omega}(s_i)\to U_{-\Omega}(s_{i+1})$. 
Taking the union of 
 these bijections for all $i$ gives the bijection $\eta$ from Proposition~\ref{pairing summary}.

Let $v$ be in $U_{\Omega}(s_i)$. Since $v$ is $c[i+1, i+n-1]$-sortable, and $J(v) = \br{s_i}$, Lemma~\ref{easy lem} says that $v = c[i+1, i+n-1] u$ for some $c[i+1, i+n-2]$-sortable element~$u$. 
Let $u'$ be the unique $c[i+n-2, i+1]$-sortable so that $\Cone_{c[i+1, i+n-2]}(u) = - \Cone_{c[i+n-2, i+1]}(u')$.
 (This exists by Proposition~\ref{double finite}). 
 We define
\[
\eta\left(c[i+1, i+n-1] \cdot u\right) := c[i+n, i+2] \cdot  (u')^{++} .
\]
Note that $(u')^{++}$ is $c[i+n,i+3]$-sortable, so $\eta(v)$ is $c[i+n, i+2]$-sortable.

\begin{prop}\label{eta bij}
The map $\eta$ is a bijection from $U_{\Omega}(s_i)$ to $U_{-\Omega}(s_{i+1})$.
\end{prop}
\begin{proof}
The map $v \mapsto u=c[i+1, i+n-1]^{-1} v$ is a bijection between $U_{\Omega}(s_i)$ and the set of $c[i+1, i+n-2]$-sortable elements; the map $u \mapsto u'$ is a bijection between $c[i+1, i+n-2]$-sortable elements and $c[i+n-2, i+1]$-sortable elements; the map $u' \mapsto (u')^{++}$ is a bijection between $c[i+n-2, i+1]$-sortable elements and $c[i+n, i+3]$-sortable elements; the map $w \mapsto c[i+n, i+2] w$ is a bijection between $c[i+n, i+3]$-sortable elements and $U_{-\Omega}(s_{i+1})$.

The map $\eta$ composes these four bijections, and is thus a bijection.
\end{proof}

The bijection $\eta$ is constructed in order to make the following result hold:
\begin{prop}\label{challenging}
For $i\in[n]$ and $v\in U_{\Omega}(s_i)$, we have ${\left| C_{\Omega}(v) \cup ( - C_{-\Omega}(\eta(v)) ) \right| = n}$ and $\Cone_{\Omega}(v) \cap \partial \Tits(A) = - \Cone_{\Omega}(\eta(v)) \cap \partial \Tits(A)$.
\end{prop}

The proof uses several lemmas.
\begin{lemma}\label{ge n-2}
For $i\in[n]$ and $v\in U_{\Omega}(s_i)$, the sets $\set{C_\Omega^{s_j}(v):i+1\le j\le i+n-2}$ and
$\set{-C_{-\Omega}^{s_j}(\eta(v)):i+3\le j\le i+n}$ coincide.
\end{lemma}
\begin{proof}
If $i+1\le j\le i+n-2$, then $c[i+1,i+n-1]\alpha_j=\alpha_{j+1}=\alpha_j^+$.
Thus,  
\begin{multline*}
\set{C_{c[i+1,i+n-1]}^{s_j}(v):i+1\le j\le i+n-2} = c[i+1, i+n-1] C_{c[i+1, i+n-2]}(u)  \\
= C_{c[i+1,i+n-2]}(u)^+  = C_{c[i+2,i+n-1]}(u^+).
\end{multline*}  
Similarly,  
\begin{multline*} 
\set{-C_{c[i+n,i+2]}^{s_j}(\eta(v)):i+3\le j\le i+n} = - c[i+n, i+2] C_{c[i+n,i+3]}((u')^{++})  \\
= - [C_{c[i+n, i+3]}((u')^{++})]^{-} = - C_{c[i+n-1, i+2]}((u')^{+}).
\end{multline*}

Because $\Cone_{c[i+1, i+n-2]}(u) = - \Cone_{c[i+n-2, i+1]}(u')$, also $C_{c[i+1, i+n-2]}(u)$ is equal to $-C_{c[i+n-2, i+1]}(u')$, so 
$ C_{c[i+2,i+n-1]}(u^+) = - C_{c[i+n-1, i+2]}((u')^{+})  $
as desired.
\end{proof}

\begin{lemma}\label{Funny Roots} 
If $i\in[n]$ and $v\in U_{\Omega}(s_i)$, then the root $-C^{s_{i+2}}_{c[i+n,i+2]}(\eta(v))$ is positive, the root $C^{s_{i-1}}_{c[i+1,i+n-1]}(v)$ is negative, and $-C^{s_{i+2}}_{c[i+n,i+2]}(\eta(v))-C^{s_{i-1}}_{c[i+1,i+n-1]}(v) = \delta$. 
\end{lemma}

\begin{proof} 
Write the first block of the $c[i+1, i+n-2]$-sorting word for $u$ in the form $s_{i_1} s_{i_2} \cdots s_{i_p} s_{g} s_{g+1} \cdots s_{i+n-2}$ where $i_p \leq g-2$. (If $s_{i+n-2}$ does not occur in $u$, then $g=i+n-1$.)
Here $i+1 \leq g \leq i+n-2$. 
We calculate
\begin{multline*}
 C^{s_{i-1}}_{c[i+1,i+n-1]}(v)= 
 c[i+1, i+n-1] s_{i_1} s_{i_2} \cdots s_{i_p} s_{g} s_{g+1} \cdots s_{i+n-2} \alpha_{i+n-1} =\\
 c[i+1, i+n-1] \left( \alpha_g + \alpha_{g+1} + \cdots + \alpha_{i+n-1} \right) = 
 - \alpha_{i+1} - \alpha_{i+2} - \cdots - \alpha_g. 
 \end{multline*} 
Similarly, write the first block of the $c[i+n-2, i+1]$-sorting word for $u'$ in the form $s_{j_1} s_{j_2} \cdots s_{j_q} s_{g'} s_{g'-1} \cdots s_{i+1}$ with $j_q \geq g'+2$. We have
\begin{multline*} 
C^{s_{i+2}}_{c[i+n,i+2]}(\eta(v)) =
c[i+n, i+2]  s_{j_1+2} s_{j_2+2} \cdots s_{j_q+2} s_{g'+2} s_{g'+1} \cdots s_{i+3} \alpha_{i+2} = \\
c[i+n, i+2] \left( \alpha_{g'+2} + \alpha_{g'+1} + \cdots + \alpha_{i+2} \right) =
- \alpha_{i+n} - \alpha_{i+n-1} - \cdots - \alpha_{g'+2} .
\end{multline*}

In Section~\ref{long computation}, we will prove that $g'=g-1$ (Proposition~\ref{prop reversal}). This computation is elementary but lengthy, which is why we have given it its own section. Accepting this result for now, 
$-C^{s_{i+2}}_{c[i+n,i+2]}(\eta(v))-C^{s_{i-1}}_{c[i+1,i+n-1]}(v)$ equals
\[ (\alpha_{i+n} + \alpha_{i+n-1} + \cdots + \alpha_{g+1} ) + (  \alpha_{i+1} + \alpha_{i+2} + \cdots + \alpha_g) = \delta. \qedhere \]
\end{proof}

\begin{proof}[Proof of Proposition~\ref{challenging}]  
The label $C_\Omega^{s_i}(v)$ is undefined and the other $n-1$ labels on $v$ are defined.
Similarly, $C_{-\Omega}^{s_{i+1}}(\eta(v))$ is undefined and $n-1$ labels on $\eta(v)$ are defined.
Thus $\left| C_{\Omega}(v) \cup ( - C_{-\Omega}(\eta(v)) ) \right| = n$ is equivalent to $\left| C_{\Omega}(v) \cap ( - C_{-\Omega}(\eta(v)) ) \right| = n-2$.
Lemma~\ref{ge n-2} implies that $\left| C_{\Omega}(v) \cap ( - C_{-\Omega}(\eta(v)) ) \right| \ge n-2$.
Lemma~\ref{Funny Roots} says in particular that $C^{s_{i-1}}_{c[i+1,i+n-1]}(v)$ and $-C^{s_{i+2}}_{c[i+n,i+2]}(\eta(v))$ are distinct, and we conclude that $\left| C_{\Omega}(v) \cup ( - C_{-\Omega}(\eta(v)) ) \right| = n$.

These two lemmas also show that the roots normal to the walls of $\Cone_{\Omega}(v)$, and the roots normal to the walls of $-\Cone_{\Omega}(\eta(v))$, differ only by adding $\delta$ to one root.
Therefore, both sets of roots define the same cone in $\delta^{\perp}$. 
Since ${\partial(\Tits(A)) = \delta^{\perp}}$, we deduce that $\Cone_{\Omega}(v) \cap \partial \Tits(A) = - \Cone_{\Omega}(\eta(v)) \cap \partial \Tits(A)$.
\end{proof}

\def\lc{\vec{c}}

\section{Proof that $g'=g-1$} \label{long computation}
We now verify the equation $g'=g-1$ needed in the proof of Lemma~\ref{Funny Roots}.
We begin by stripping down the notation.
Let $t_1$, $t_2$, \dots, $t_{k-1}$ be the simple generators of the symmetric group $S_k$, numbered in order along the Dynkin diagram.
 Let $S_{[1,k-1]}$ be the subgroup of $S_k$ generated by $t_1$, $t_2$, \ldots, $t_{k-2}$ and let $S_{[2,k]}$ be the subgroup generated by $t_2$, $t_3$, \dots, $t_{k-1}$. Let $\lc$ denote the Coxeter element $t_1 t_2 \cdots t_{k-1}$.
When we apply the results of this section to prove Lemma~\ref{Funny Roots}, we will have $k=n-2$ and $t_j = s_{i+j}$.

\begin{lemma} \label{pidown coset}  
For any $w \in S_k$, the $\lc$-sortable element $\pidown^{\lc}(w)$ is in the coset~$w S_{[1,k-1]}$.
\end{lemma}

\begin{proof}
The left cosets of $S_{[1,k-1]}$ are intervals in the weak order on $S_k$.
The minimal elements of these intervals are products of the form $t_g t_{g+1} \cdots t_{k-1}$, for $1 \leq g \leq k$.  
These minimal elements are $\lc$-sortable. 

Choose $g$ so that $w$ is in the coset $t_g t_{g+1} \cdots t_{k-1} S_{[1,k-1]}$. Then $w \geq t_g t_{g+1} \cdots t_{k-1}$ so $w \geq \pidown^{\lc}(w) \geq \pidown^{\lc}(t_g t_{g+1} \cdots t_{k-1}) = t_g t_{g+1} \cdots t_{k-1}$. Since $w$ and $t_g t_{g+1} \cdots t_{k-1}$ are in the same $S_{[1,k-1]}$ coset, we deduce that $\pidown^{\lc}(w)$ also lies in that coset.
\end{proof} 

\begin{lemma} \label{coset block1} 
Suppose $u$ is ${\lc}$-sortable and lies in $t_g t_{g+1} \cdots t_{k-1} S_{[1,k-1]}$. 
The first block of the ${\lc}$-sorting word for $u$ is of the form $t_{i_1} t_{i_2} \cdots t_{i_r} t_g t_{g+1} \cdots t_{k-1}$ where $i_r \leq g-2$.
\end{lemma}

\begin{proof}  
Since $u \in t_g t_{g+1} \cdots t_{k-1} S_{[1,k-1]}$, we know that $u$ has a reduced word which begins $t_g t_{g+1} \cdots t_{k-1}$. 
In particular, the generators $t_g$, $t_{g+1}$, \dots, $t_{k-1}$ are contained in any reduced word for $u$, so the first block of the $\lc$-sorting word for $u$ contains these generators. 
We now need to show that $t_{g-1}$ is \emph{not} in this leading block. 
Suppose, for the sake of contradiction, that the leading block of the $\lc$-sorting word 
contains $t_f t_{f+1} \cdots t_{k-1}$ and not $t_{f-1}$, for some $f<g$. Then $u \geq t_f t_{f+1} \cdots t_{k-1}$. 
But the upper order ideal of elements above $t_f t_{f+1} \cdots t_{k-1}$ is disjoint from the interval $t_g t_{g+1} \cdots t_{k-1} S_{[1,k-1]}$.  
\end{proof}

Recall that $D$ is the fundamental domain $\{ x \in V^{\ast} : \langle \alpha_i, x \rangle  \geq 0 \}$.
\begin{lemma} \label{w0 negation}
Let $x \in S_k$ lie in the coset $t_g t_{g+1} \cdots t_{k-1} S_{[1,k-1]}$. Let $x'$ be the unique element of $S_k$ so that $x D = - x' D$. Then $x'$ lies in the coset $t_{g-1} \cdots t_2 t_1  S_{[2,k]}$.
\end{lemma}

\begin{proof}
We have $x'=x w_0$. We compute
\[ t_g t_{g+1} \cdots t_{k-1} S_{[1,k-1]} w_0 = t_g t_{g+1} \cdots t_{k-1} w_0 S_{[2, k]} = t_{g-1} \cdots t_2 t_1 S_{[2,k]}  \]
where the second equality is a straightforward computation.
\end{proof}

We combine the above lemmas to prove:

\begin{prop} \label{prop reversal}
Suppose $u$ is $\lc$-sortable, suppose $u'$ is $\lc\,^{-1}$-sortable, and suppose that $\Cone_{\lc}(u) = -\Cone_{\lc\,^{-1}}(u')$. 
If $t_{i_1} t_{i_2} \cdots t_{i_r} t_g t_{g+1} \cdots t_{k-1}$ is the first block of the $\lc$-sorting word for $u$ (with $i_r \leq g-2$), and $t_{j_1} t_{j_2} \cdots t_{j_s} t_{g'} t_{g'-1} \cdots t_1$ is the first block of the $\lc\,^{-1}$-sorting word for $u'$ (with $j_s \geq g'+2$), then $g'=g-1$.
\end{prop}

\begin{proof}
By Lemma~\ref{coset block1} and its image under the automorphism $t_i \mapsto t_{k-i}$, the permutations $u$ and $u'$ lie in the cosets $t_g t_{g+1} \cdots t_{k-1} S_{[1,k-1]}$ and $t_{g'} \cdots t_2 t_1 S_{[2,k]}$.
Write $x'$ for the permutation so that $-u D = x' D$, so Lemma~\ref{w0 negation} shows that $x'$ is in the coset $t_{g-1} \cdots t_2 t_1  S_{[2,k]}$.
We have $x' D = - u D \subseteq - \Cone_{\lc}(u) = \Cone_{\lc^{-1}}(u')$ so, by Lemma~\ref{pidown coset}, the permutation $u'$ also lies in the coset $t_{g-1} \cdots t_2 t_1  S_{[2,k]}$.
We see that $g'=g-1$, as desired.
\end{proof}

\begin{remark}
The results of this section can be given alternate proofs using the combinatorial descriptions of type-$A$ Cambrian combinatorics from~\cite{cambrian}. 
The reader who wishes to attempt this should note that, if we write $w$ in one line notation as $a_1 a_2 \cdots a_k$, then $u$ lies in the cosets $t_{a_k} t_{a_k+1} \cdots t_{k-1} S_{[1, k-1]}$ and $t_{a_1-1} \cdots t_2 t_1 S_{[2,k]}$.
\end{remark}

\section{Proof of Theorem~\ref{main theorem}}\label{proofs sec}
In this section, we complete the proof of Theorem~\ref{main theorem}, which asserts that the pair  
$(\DCamb_{\Omega}, \DC_{\Omega})$ is a complete descending framework.
As discussed in Section~\ref{frameworks sec}, this also completes the proof of Theorem~\ref{cycle thm}.

\begin{proof}[Proof of Theorem~\ref{main theorem}]
The Co-label and Base conditions are obviously satisfied, and the Sign condition holds because the labels are roots, 
so we need to check the Transition Condition to establish that $(\DCamb_\Omega, \DC_\Omega)$ is a framework.  
That check is the hardest part of the proof and we leave it to the end.

Once the Transition condition is checked, the framework is complete by construction.
We now show that it is a descending framework.
To see that the Unique minimum condition holds, we need to show that every $v$ in $\DCamb_\Omega$ has a negative label, except for the identity in $\Camb_\Omega$.
First, suppose $v$ is a vertex of $\Camb_\Omega$.
If $v$ is not the identity, then it has at least one cover reflection, and thus Proposition~\ref{lower walls} says that $v$ has at least one negative label.
Now, suppose that $v$ is a vertex of $-\Camb_{-\Omega}$. By Proposition~\ref{TwoTypes}, $v$ lies in some parabolic $W_{\br{r}}$ and thus Lemma~\ref{sign omit} implies that $v$ has a negative label.  

The previous paragraph establishes the Unique minimum condition.
The Full edge condition holds because the framework is complete.
The Descending chain condition holds because neither $\Camb_\Omega$ nor $-\Camb_{-\Omega}$ has any infinite chains.
Thus, to show that $(\DCamb_\Omega, \DC_\Omega)$ is a descending framework, we must check the Transition condition.  
To check this, we make use of the symmetry of the Transition condition and the antipodal symmetry of switching $\Omega$ with $-\Omega$ while applying the antipodal map.

Let $e$ be an edge connecting vertices $v$ and $w$, with $\DC_\Omega(v,e) = \beta$.
We first check that $- \beta = \DC_{\Omega}(w,e)$. 
Either $v$ and $w$ are both in $\Camb_{\Omega}$, or both in $-\Camb_{\Omega}$.
By the antipodal symmetry, we assume the former. 
By the symmetry of the Transition condition, we may assume that $v \geq w$, so $\beta$ is a negative root.
Since $\beta$ is a negative root, it is not the extra root added to $C_{\Omega}(v)$ to make $\DC_\Omega(v)$ (see Lemma~\ref{Funny Roots}). 
So $\beta = C^{s_p}_{c(\Omega,J(v) \cup \{ s_p \})}(v)$ for some $s_p$.
Moreover, since $\beta$ is a negative root, we have $s_p \in J(v)$  (see Lemma~\ref{sign omit}).   
Since $v \in W_{J(v)}$ and $w \leq v$, we have $w \in W_{J(v)}$. 
Writing $c$ for $c(\Omega,J(v))$, the Cambrian framework $(\Camb_c,C_c)$ is a framework, so $- \beta \in C_c(w)$ and, by Lemma~\ref{observe}, $-\beta \in \DC_\Omega(w)$.
This trick of reducing to an acyclic parabolic will be used repeatedly.

Let $\gamma \in \DC_\Omega(v)$, with $\gamma \neq \beta$. 
For brevity, when $J\subsetneq S$, we will refer to ``the Transition condition in the Cambrian framework $(\Camb_{c(\Omega,J)},C_{c(\Omega,J)})$" as ``the Transition condition in $W_J$.''

We must check that $\gamma+[\sgn(\beta)\omega(\beta,\gamma)]_+\,\beta$ is in $\DC_\Omega(w)$.

\noindent
\textbf{Case 1.} $|J(v)|$, $|J(w)| \leq n-2$.  
Either $v$ and $w$ are both in $\Camb_{\Omega}$ or both in $\Camb_{-\Omega}$, and by the antipodal symmetry we may assume the former. 
By the symmetry of the Transition condition, we may assume $v \geq w$. 
Then $w$ and $v$ are in the acyclic parabolic $W_{J(v)}$ and $\gamma = C^{s_r}_{c(J(v) \cup \{ s_r \} )}$ for some $r$.  
The claim follows from Lemma~\ref{observe} and the Transition condition in $W_{J(v) \cup \{ s_r \}}$.

\noindent
\textbf{Case 2.} $|J(v)|=|J(w)|= n-1$. We have $J(v)=J(w)$ since, if $v \leq w$, then $J(v) \subseteq J(w)$ and similarly if $v \geq w$. Set $J = J(v) = J(w)$. By the antipodal symmetry, we may assume that $\gamma = C^{s_r}_{c(J)}(v)$ for some $s_r \in J$. Then the claim follows from Lemma~\ref{observe} and the Transition condition in $W_{J}$.

\noindent 
\textbf{Case 3.} One of $|J(v)|$, $|J(w)|$ is $n-1$ and the other is not.
By the antipodal symmetry, we may assume that $v,w \in \Camb_{\Omega}$.
By the symmetry of the Transition condition, we may assume $|J(v)| = n-1$, so that $v \geq w$ and $\beta$ is a negative root.
Thus $v$ and $w$ are in $W_{J(v)}$. 

\noindent
\textbf{Case 3a.} $\gamma$ is of the form $C^{s_r}_{c(\Omega,J(v))}$ 
for some $s_r$ in $J(v)$. Then the claim follows from the Transition condition in $W_{J(v)}$ as before. 

\noindent
\textbf{Case 3b.} We are left with the only hard case:
We have $w\le v$ in $\Camb_{\Omega}$ with $|J(v)| = n-1$ and with $| J(w)| \leq n-2$, and furthermore $\gamma$ is the lone element of $\DC_\Omega(v) \setminus C_{J(v)}(v)$. 
Choose $s_i$ so that $J(v) = \br{s_i}$ and thus $\gamma = -C^{s_{i+2}}_{c[i+n,i+2]}(\eta(v))$.
So $v$ and $w$ are $c[i+1, i+n-1]$-sortable elements.

In this paragraph, we will show that $\beta = C^{s_{i-1}}_{c[i+1,i+n-1]}(v)$.
By Proposition~\ref{TwoTypes}, the cone $\Cone_\Omega(v)$ intersects both $\Tits(A)$ and $- \Tits(A)$ but $\Cone_\Omega(w)$ lies in 
the closure of $\Tits(A)$. 
Thus, the facet $\Cone_\Omega(w) \cap \Cone_\Omega(v) = \beta^{\perp} \cap \Cone_\Omega(v)$ lies in the closure of $\Tits(A)$. 
By Lemma~\ref{Funny Roots}, $\partial \Tits(A) = \delta^{\perp}$ contains the \mbox{codimension-$2$} face $F$ of $\Cone_\Omega(v)$ lying in $C^{s_{i-1}}_{c[i+1,i+n-1]}(v)^{\perp}$ and $-C^{s_{i+2}}_{c[i+n,i+2]}(\eta(v))^{\perp}$. 
The $n-2$ rays of $\Cone_\Omega(v)$ contained in $F$ are in $\partial\Tits(A)$, but since $\Cone_\Omega(v)$ intersects both $\Tits(A)$ and $- \Tits(A)$, one of the other two rays is in $\Tits(A)$ while the other is in $-\Tits(A)$.
Therefore, all facets of $\Cone_\Omega(v)$ cross $\partial\Tits(A)$ except the facets defined by $C^{s_{i-1}}_{c[i+1,i+n-1]}(v)$ and $-C^{s_{i+2}}_{c[i+n,i+2]}(\eta(v))$, so that $\beta$ is one of these two roots.
But $\beta$ is negative, so by Proposition~\ref{Funny Roots} it must be $C^{s_{i-1}}_{c[i+1,i+n-1]}(v)$.

Write $v$ as $c[i+1, i+n-1] u$ and suppose the first block of the ${c[i+1, i+n-2]}$-sorting word for $u$ is  $s_{i_1} s_{i_2} \cdots s_{i_p} s_{g} s_{g+1} \cdots s_{i+n-2}$ as in the proof of Lemma~\ref{Funny Roots}. 
As in that proof, $\beta = - \alpha_{i+1} - \alpha_{i+2} - \cdots - \alpha_g$ and $\gamma = \alpha_{g+1} + \cdots + \alpha_{i+n}$. We compute that $\omega(\beta, \gamma) = 0$, so our goal is to show that $\gamma \in \DC_\Omega(w)$. 
Specifically, we will show that $\gamma = C^{s_i}_{J(w) \cup \{ s_i \}}(w)$. 

Our first task is to check that $s_g \not \in J(w)$. 
Set $t_r = s_{i+1} s_{i+2} \cdots s_r \cdots s_{i+2} s_{i+1}$ for each $r=i+1,\ldots,i+n-2$.
In particular, $\beta$ is the negative root associated to $t_g$. 
The $t_r$, for $i+1 \leq r \leq i+n-2$, are inversions of $v$. 
Proposition~\ref{lower walls} says $t_r$ is a cover reflection of $v$, so $t_{i+1}$, $t_{i+2}$, \dots, $t_{g-1}$ are inversions of $t_g v$ and $t_g$ is not. 
We deduce that $t_g v \geq c[i+1, g-1]$ and $t_g v \not \geq c[i+1, g]$. 
Since $\Cone_\Omega(v)$ and $\Cone_\Omega(w)$ share a facet contained in $\beta^\perp$ and since $vD$ is in $\Cone_\Omega(v)$, we see that $t_gvD$ is in $\Cone_\Omega(w)$.
Thus Theorem~\ref{pidown fibers} says that $w=\pidown^{\Omega}(t_g v)$.  
Since $c[i+1, g-1]$ and $c[i+1, g]$ are both $c[i+1,i+n-1]$-sortable, Lemma~\ref{helpful} and the relations $t_g v \geq c[i+1, g-1]$ and $t_g v \not \geq c[i+1, g]$ now imply that $w\geq c[i+1,g-1]$ and $w \not \geq  c[i+1, g]$. 
Therefore, the first block of the $c[i+1, i+n-1]$-sorting word of $w$ does not contain $s_g$, and thus $s_g \not \in J(w)$ because $w$ is $c[i+1, i+n-1]$-sortable. 

We have shown that $J(w) \cup \{ s_i \} \subseteq \br{s_g}$, so $C^{s_i}_{J(w) \cup \{ s_i \}}(w) = C^{s_{i+n}}_{c[g+1, g+n-1]}(w)$. 
(We changed $s_i$ to $s_{i+n}$ for clarity, as $g+2 \leq i+n \leq g+n-1$.)
We define $u_r = s_{g+1} s_{g+2} \cdots s_r \cdots s_{g+2} s_{g+1}$ for each $r=g+1,\ldots,i+n-1$ and note that 
\[ u_r = \left( c[i+1, i+n-1] s_g s_{g+1} \cdots s_{r-2} \right) s_{r-1} \left( c[i+1, i+n-1] s_g s_{g+1} \cdots s_{r-2} \right)^{-1}. \]
Then $u_r$ is an inversion of $v$ for $g+1 \leq r \leq i+n-1$. 
 Thus, $u_r$ is also an inversion of $t_g v$ for $r$ in this range, and we deduce that $t_g v \geq c[g+1, i+n-1]$. 
 Using once again that $\pidown^{\Omega}$ is order preserving, this implies that $w=\pidown^\Omega(t_gv) \geq c[g+1, \ldots, i+n-1]$. 
 On the other hand, $s_i \not \in J(w)$. 
 So the $c[g+1, g+n-1]$-sorting word for $w$ begins $s_{g+1} s_{g+2} \cdots s_{i+n-1}$, before skipping $s_{i+n}$. 
 This shows that $C^{s_{i+n}}_{c[g+1, g+n-1]}(w) =  \alpha_{g+1} + \cdots + \alpha_{i+n}= \gamma$ as promised. 
 \end{proof}

\subsection*{Acknowledgments} 
We thank Salvatore Stella for helpful comments and for pointing out an error in an early version of this manuscript.
We also thank Christophe Hohlweg for helpful comments, and for pointing out the observation made in Remark~\ref{Shi remark}.


\begin{thebibliography}{18}

\bibitem{BjBr}
A. Bj\"{o}rner and F. Brenti,
\textit{Combinatorics of Coxeter groups.}
Graduate Texts in Mathematics, \textbf{231},
Springer, New York, 2005. 

\bibitem{Bour}
N. Bourbaki, \textit{Lie groups and Lie algebras. Chapters 4-6.}
Translated from the 1968 French original by A. Pressley, 
Graduate Texts in Mathematics, Springer-Verlag, 2002.

\bibitem{ga}
S. Fomin and A. Zelevinsky,
\textit{$Y$-systems and generalized associahedra.}
Ann. of Math. (2) \textbf{158} (2003), no. 3, 977--1018.

\bibitem{ca2}
S. Fomin and A. Zelevinsky,
\textit{Cluster Algebras II: Finite Type Classification.}
Inventiones Mathematicae \textbf{154} (2003), 63--121.

\bibitem{ca4}
S. Fomin and A. Zelevinsky,
\textit{Cluster Algebras IV: Coefficients.}
Compositio Mathematica \textbf{143} (2007), 112--164.

\bibitem{GHKK}
M. Gross, P. Hacking, S. Keel, and M. Kontsevich,
\textit{Canonical bases for cluster algebras.}
Preprint, 2014. \arxiv{1411.1394}

\bibitem{Humphreys}
J.~Humphreys,
\textit{Reflection Groups and Coxeter Groups.}
Cambridge Studies in Advanced Mathematics {\bf 29},
Cambridge Univ. Press, 1990.

\bibitem{Kac}
V. Kac,
\emph{Infinite-dimensional Lie algebras.}
Third edition. Cambridge University Press, Cambridge, 1990.

\bibitem{cambrian}
N.~Reading, 
\textit{Cambrian Lattices.}
Adv. Math. \textbf{205} (2006), no.~2, 313--353.

\bibitem{sortable}
N.~Reading,
\emph{Clusters, Coxeter-sortable elements and noncrossing partitions.}
Trans. Amer. Math. Soc. \textbf{359} (2007), no. 12, 5931--5958.

\bibitem{sort_camb}
N.~Reading,
\emph{Sortable elements and Cambrian lattices.}
Algebra Universalis \textbf{56} (2007), no. 3-4, 411--437. 

\bibitem{universal}
N.~Reading,
\textit{Universal geometric cluster algebras.}
Math. Z. \textbf{277} (2014), no. 1--2, 499--547.

\bibitem{camb_fan}
N.~Reading and D.~Speyer,
\emph{Cambrian Fans.} 
J. Eur. Math. Soc. (JEMS) \textbf{11} (2009), no. 2, 407--447. 

\bibitem{typefree}
N.~Reading and D.~Speyer,
\emph{Sortable elements in infinite Coxeter groups.} 
Trans. Amer. Math. Soc. \textbf{363} (2011) no. 2, 699--761. 

\bibitem{cyclic}
N.~Reading and D.~Speyer,
\emph{Sortable elements for quivers with cycles.}
Electron. J. Combin. \textbf{17(1)} (2010), Research Paper 90, 19 pp. 

\bibitem{framework}
N.~Reading and D.~Speyer,
\emph{Combinatorial frameworks for cluster algebras.}
Int. Math. Res. Not., published online May 11, 2015. (Full citation information pending.)

\bibitem{afframe}
N.~Reading and D.~Speyer,
\emph{Cambrian frameworks for cluster algebras of affine type.}
Preprint, 2015. \arxiv{1504.00260} 

\bibitem{Shi}
J.-Y.~Shi, 
\emph{The Kazhdan-Lusztig cells in certain affine Weyl groups.}
Lecture Notes in Mathematics \textbf{1179}. Springer-Verlag, Berlin, 1986. 

\end{thebibliography}
\end{document}